\documentclass[12pt]{amsart}
\usepackage{amscd,epsfig,fullpage,url}
\usepackage[colorlinks=true, pdfstartview=FitV, linkcolor=blue, citecolor=blue, urlcolor=blue]{hyperref}
\usepackage{amssymb}
\usepackage{comment}
\usepackage{graphicx}
\usepackage{tikz-cd} 

\usepackage[normalem]{ulem}

\title[Bracket width of simple Lie algebras]{Bracket width of simple Lie algebras}
\author{Adrien Dubouloz, Boris Kunyavski\u\i \  and  Andriy Regeta}

\address{IMB UMR5584, CNRS, Universit\'e Bourgogne Franche-Comt\'e, \newline \indent F-21000 Dijon, France.}
\email{adrien.dubouloz@u-bourgogne.fr}

\address{Department of Mathematics, Bar-Ilan University, \newline \indent Ramat Gan 5290002, Israel.}
\email{kunyav@gmail.com }

\address{\noindent Institut f\"{u}r Mathematik, Friedrich-Schiller-Universit\"{a}t Jena, \newline
\indent  Jena 07737, Germany.}
\email{andriyregeta@gmail.com}

\thanks{ The first author was partially supported by the French ANR grant FIBALGA ANR-18-CE40-0003-01. Research of the second author was supported by the ISF grants 1623/16 and 1994/20. This research started during the visit of the second author to the MPIM (Bonn) and continued during the visit of the third author to Bar-Ilan University. The support of these institutions is gratefully acknowledged.}

\newtheorem{question}{Question}
\newtheorem{theorem}{Theorem}
\newtheorem{corollary}{Corollary}
\newtheorem{lemma}{Lemma}
\newtheorem{proposition}{Proposition}
\newtheorem{definition and proposition}{Definition and proposition}
\newtheorem{conjecture}{Conjecture}

\newtheorem*{hyp}{Hypothesis (J)}

\newtheorem*{thA}{Theorem A}
\newtheorem*{thB}{Theorem B}
\theoremstyle{definition}
\newtheorem{definition}{Definition}
\newtheorem{remark}{Remark}
\newtheorem{example}{Example}

\newcommand{\name}[1]{\textsc{#1\/}}

\newcommand{\Der}{\operatorname{Der}}
\renewcommand{\d}{{\partial}}

\DeclareMathOperator{\Div}{Div}
\DeclareMathOperator{\Ve}{Vec}

\DeclareMathOperator{\Jac}{Jac}

\DeclareMathOperator{\ad}{ad}

\DeclareMathOperator{\Aut}{Aut}
\DeclareMathOperator{\SAut}{SAut}

\DeclareMathOperator{\Ker}{Ker}

\DeclareMathOperator{\LND}{LNV}

\DeclareMathOperator{\Imm}{Im}

\def\Dp{\mathrm{D}_p}

\def\Dz(z-1){\mathrm{D}_{z(z-1)}}
\def\Dq{\mathrm{D}_q}

\newcommand{\lielnd}[1]{\langle \LND(#1) \rangle}

\def\O{{\mathcal O}}

\def \itt #1,#2:{\medskip\item[$\bullet$] 
     page\ \ignorespaces#1, line\ \ignorespaces#2:\ \ignorespaces}

\begin{document}

\begin{abstract} The notion of commutator width of a group, defined as the smallest number of commutators needed to represent each element of the derived group as their product, has been extensively studied over the past decades. In particular, in 1992 Barge and Ghys discovered the first example of a simple group of commutator width greater than one among groups of diffeomorphisms of smooth manifolds.  

We consider a parallel notion of bracket width of a Lie algebra and present the first examples of simple Lie algebras of bracket width greater than one. They are found among the algebras of algebraic vector fields on smooth affine varieties. 
\end{abstract}

\maketitle

%%%%%%%%%%%%%%%%%%%%%%%%%%%%%%%%%%%%%%%%%

\section{Introduction}

The notion of width in various contexts appears in many algebraic structures. It is particularly well studied in group theory where the {\it commutator width} got special attention, being related to many important properties of the class of groups under investigation. Recall that given a group $G$, the commutator width is defined as supremum of the lengths $\ell (g)$, $g$ running over the derived subgroup $[G,G]$, where $\ell (g)$ is the smallest number of commutators needed 
to represent $g$ as their product. Examples of groups of  commutator width greater than one were known to Miller more than 100 years ago, and such examples can easily be discovered even among perfect groups (i.e., groups $G$ coinciding with $[G,G]$). However, it took quite a time to discover {\it simple} groups of commutator width greater than one. As often happens, the insight came from outside: \name{Barge} and \name{Ghys} \cite{BG} have found their examples among groups of differential-geometric nature, applying deep results from that area. The groups they considered were infinite (in any {\it finite} simple group every element is a single commutator, as predicted by Ore's conjecture settled in 2010 \cite{LOST}, after more than half-a-century efforts). The interested reader can find more details in the survey \cite{GKP} and the references therein. 

The group case served for us as a prototypical example for  investigating the similar notion of {\it bracket width} in  the parallel universe of Lie algebras. It is defined in full analogy with the commutator width of a group. Given a Lie  algebra $L$ over a field $k$, we define its bracket width  as supremum of the lengths $\ell (a)$, $a$ running over the derived algebra $[L,L]$, where $\ell (a)$ is defined as the smallest number $m$ of Lie brackets $[x_i,y_i]$ needed to 
represent $a$ in the form 
$$
a=\sum_{i=1}^m[x_i,y_i]. 
$$
The bracket width applies in studying different aspects of Lie algebras, ranging from problems motivated by logic (elementary equivalence), see \cite{Rom}, to those coming from differential geometry, see \cite{LT}. Our focus is on {\it simple} Lie algebras. Here one can observe certain parallelism with the commutator width of groups: there is little hope to find  an example of a {\it finite-dimensional} 
simple Lie algebra of bracket width greater than one, though in general the problem is open (there are many cases where 
the width is known to be equal to one, and it is known that 
it cannot exceed two \cite{BN}); see \cite[Section~6]{GKP} 
for more details.   

Examples of simple Lie algebras of width greater than one should thus be sought among {\it infinite-dimensional} algebras. There are several natural families of such to be checked first. Four families of Lie algebras of Cartan type all consist of the algebras of width one, in light of the results of \name{Rudakov} \cite{Rud}. The case of (subquotients of) 
Kac--Moody algebras is open, to the best of our knowledge. In the present paper, we start the study of another classically known family, algebras $\Ve(X)$ of algebraic vector fields on  irreducible affine varieties $X$.  
It is well-known that 
the Lie algebra $\Ve(X)$ is simple if and only if $X$ is smooth (see \cite{Jor}, \cite[Proposition~1]{Sie96}). Thus our primary objects of interest are smooth affine curves and surfaces. 

It turns out that already among the Lie algebras of vector fields on smooth affine curves there are algebras of width greater than one.  Our first main result is the following:

\begin{thA} (Theorem \ref{main2}) 
Let $C$ be an irreducible smooth affine non-rational curve with a unique place at infinity and trivial tangent sheaf. Then the bracket width of the Lie algebra $\mathrm{Vec}(C)$ is larger than or equal to two. 
\end{thA}

There is an abundance of affine curves with only one place at infinity, and they were studied in various contexts, see, e.g., a recent paper of \name{Koll\'ar} \cite{Kol} and the references therein. 
A simple class of examples is given by affine hyperelliptic curves plane curves $C \subset \mathbb{A}^2$ defined by equations of the form  $y^2=h(x)$, where  $h(x)$ is a separable monic polynomial of odd degree greater than one (Example \ref{example-hyper-elliptic}), whose algebras $\Ve(C)$ have been studied by \name{Billig} and \name{Futorny}  \cite{BF}. More examples can be found at the end of Section \ref{sec:curves} below. 

\medskip

In the two-dimensional case, we consider smooth affine surfaces endowed with an algebraic volume form. In 
complex-analytic set-up, such varieties have been intensively studied in particular by \name{Kaliman} and \name{Kutzschebauch} in connection to the algebraic volume density property, see \cite{KK10}, \cite{KK16} and the references therein. Note that algebraic volume forms naturally arise once the $2n$-dimensional variety under consideration carries a symplectic form $\omega$ because in this case the $n^{\textrm{th}}$ exterior power of $\omega$ is a volume form. Looking at a symplectic surface $(S,\omega)$, we  
consider certain natural Lie subalgebras of vector fields on $S$, related to Hamiltonian vector fields, with a goal 
to establish the simplicity and compute the bracket width.  
Note that the Lie algebra of Hamiltonian vector fields on $S$ 
can be identified with the Lie algebra arising from the 
natural Poisson bracket on $\mathcal O(S)/k$, the algebra of 
regular functions on $S$ modulo scalars. This observation,  
allowing one to translate questions and statements from 
Poissonian to Hamiltonian and vice versa, is essential for 
some of our proofs.

In the present paper, we mainly focus on the particular class of  so-called Danielewski surfaces
$\Dp \subset \mathbb{A}^3$ given by the equation $xy = p(z)$ where $p$ is a polynomial without multiple roots.  
These surfaces attracted a lot of attention over the last few 
decades. They were initially used in \cite{Dan89} to present a counterexample to a generalized version of the Zariski Cancellation Problem. 
They carry a natural algebraic volume form unique up to scalar multiplication and our main object of study is the width of the Lie subalgebra  $\lielnd{\Dp}$ of their Hamiltonian vector fields generated by all locally nilpotent vector fields. By virtue of \cite[Theorem 1]{LR19}, these Lie algebras $\lielnd{\Dp}$ are simple.

It is known if $\Dp$ and $\Dq$ are isomorphic then $\deg p=\deg q$ (see e.g., \cite[Lemma 2.10]{Da04} for a more precise result) and that the Lie algebras $\lielnd{\Dp}$ and $\lielnd{\Dq}$ are isomorphic if and only if $\Dp$ and $\Dq$ are isomorphic (see \cite[Theorem 2]{LR19}).

If $\deg p =1$, then $\Dp$ is isomorphic to the affine plane $\mathbb{A}^2$.  It is  well-known that in this case $\lielnd{\Dp}$ is isomorphic to the Lie algebra of vector fields with zero divergence, and its bracket width is equal to one (see Proposition \ref{affinespace} below for a more general statement). 
%\begin{remark}\label{Daigle}
%By \cite[Lemma 2.10]{Da04} two Danielewski surfaces $\Dp$ and $\Dq$ are isomorphic if and only if there exists an automorphism of algebras $F\colon  \mathbb{C}[z]  \xrightarrow{\sim} 
%  \mathbb{C}[z]$ such that $F(p) = cq$
%for some $c \in \mathbb{C}^*$.  In particular, this implies that $\deg p = \deg q$. If $\deg p = \deg q=2$, then $\Dp$ and $\Dq$ are isomorphic.
%\end{remark}
We next consider the cases $\deg p=2$ and $\deg p\ge 3$ separately. In the latter case we assume in addition that a certain condition on the algebra of regular functions on $\Dp$ 
holds (this is Hypothesis (J) in Section \ref{section3.1} below, whose flavour is somewhat reminiscent of the Jacobian Conjecture). Here is our next main result. 

\begin{thB} (Proposition \ref{danielewskiwidthzero} and Theorem \ref{main1})
Let $\Dp \subset \mathbb{A}^3$ be given by the equation $xy = p(z)$ where $p$ is a polynomial without  multiple roots. 
\begin{itemize}
\item[(i)] If $\deg p = 2$, then the bracket width of  $\lielnd{\Dp}$ is at most two.
\item[(ii)] If $\deg p \ge 3$, assume in addition that Hypothesis (J) holds. Then the bracket width of  $\lielnd{\Dp}$ is greater than one.
\end{itemize}
\end{thB}
In the case  $\deg p = 2$, besides the estimate of Theorem B(i), we prove that the bracket width of $\lielnd{\Dp}$ is equal to one if and only if every algebraic $2$-form on $\Dp$ satisfies 
a kind of `bi-exactness' property, see Lemma \ref{biexact} for the precise statement. 

\medskip

\noindent{\it{Notation and conventions.}} In what follows $k$ denotes an algebraically closed field of characteristic zero. By a variety, we mean an irreducible and reduced $k$-scheme of finite type. 
For a smooth variety $X$, we denote by $\mathcal{T}_X$ the tangent sheaf of $X$, that is, the dual of the coherent locally free sheaf $\Omega_{X/k}$ of K\"ahler  differentials on $X$. We denote by 
$\Ve(X)$  the Lie algebra $H^0(X,\mathcal{T}_X)$ of global vector algebraic fields $X$, endowed with its Lie bracket $[\cdot, \cdot]$.

For an affine variety $X$, we frequently identify elements of $\Ve(X)$ with $k$-derivations of the algebra $\mathcal O(X)=H^0(X,\mathcal{O}_X)$ of regular functions on $X$ via the canonical $k$-linear isomorphism $\Ve(X) \to 
\mathrm{Der}(\O(X),\O(X))$
which associates to a global vector field $\xi$ on $X$ the derivation $\mathcal{L}_\xi$ obtained as the composition of the canonical $k$-derivation   $d\colon\O(X)\to \Omega_{\O(X)/k}$ 
with the interior contraction $i_\xi\colon\Omega_{\O(X)/k}\to \O(X)$.

Throughout below by `width' of a Lie algebra we mean `bracket width',  unless specified otherwise.

\section{Some basic examples of  Lie algebras of width 1}

It is well known\footnote{See, e.g., the proof of Proposition~1 in \cite[Section~2]{Rud}} that every element $\mu = f_1 \d_{x_{i_1}} + \dots + f_k \d_{x_{i_k}}$ of the simple Lie algebra $\Ve(\mathbb{A}^n)=\mathrm{Der}_kk[x_1,\dots x_n]$ can be written as the single bracket \[\mu= [\d_{x_1}, g_1 \d_{x_{i_1}} + \dots + g_k \d_{x_{i_k}}],   \]
 where $g_i$ is a polynomial such that $\frac{\d g_i}{\d x_1} = f_i$ for every $i = 1, \dots, k$.  More generally, we have the following result. 
  
\begin{proposition}\label{A1timesX}
Let $X$ be a smooth affine variety. Then for every $n \geq 1$, $\Ve(\mathbb{A}^n \times X)$
is a simple Lie algebra of width one.
\end{proposition}

\begin{proof}
By induction on $n$, it suffices to prove the assertion in the case $n=1$. Let $\mathbb{A}^1=\mathrm{Spec}(k[x])$, let $\nu \in \Ve(X)$, and let $k \ge 0$. Then $x^m\nu \in \Ve(\mathbb{A}^1 \times X)$ can be represented as a single bracket: 
\begin{equation} [\d_x,\frac{1}{m+1}x^{m+1}\nu] =  \frac{1}{m+1}( \d_x(x^{m+1})\nu
+ x^{k+1}  [\d_x,\nu]) = x^k\nu.  
\end{equation}

The same conclusion holds for any $x^mf\d_x$, where $f \in \mathcal{O}(X)$ and $k \ge  0$: 
\begin{equation} [\d_x,\frac{1}{m+1}x^{m+1}f\d_x] =  \frac{1}{m+1}( \d_x(x^{m+1})f\d_x
+ x^{m+1}  [\d_x,f\d_x]) = x^mf\d_x.  
\end{equation}
Since every $\mu \in \Ve(\mathbb{A}^1 \times X)$  can be  represented as a sum of elements of the form $x^m\nu$, $\nu \in \Ve(X)$, $m\ge 0$, and elements of the form  $x^mf\d_x$, $f \in \mathcal{O}(X)$, $m\ge 0$, we conclude that $\mu = [\d_x,\delta]$ for a suitable $\delta \in \Ve(\mathbb{A}^1 \times X)$.
\end{proof}

The same conclusion holds for the Witt algebra $$W_n=\Ve (\mathbb T^n)=\Der k[x_1,x_1^{-1},\dots ,x_n,x_n^{-1}]$$  of vector fields on the $n$-dimensional torus $\mathbb{T}^n=\mathbb{G}_m^n$: each element of $W_n$ can be represented as a single bracket. Actually, we have the following stronger result: 

\begin{proposition}\label{A*times}
Let $X$ be a smooth affine variety. Then for every $n\geq 1$, $\Ve(\mathbb T^n \times X)$ is a simple Lie algebra of width one. 

\end{proposition} 

\begin{proof}
Again, it suffices to establish the assertion in the case 
 where $n=1$. Set $\mathbb{T}^1=\mathrm{Spec}(k[x,x^{-1}])$. For every $m\in \mathbb{Z}$ and $f \in \mathcal{O}(X)$, we have  \begin{equation}\label{torustimes}
x^mf\d_x = [x^l \d_x,  \frac{1}{m-2l+1}x^{m-l+1}f\d_x],
\end{equation}
 where $m-2l+1 \neq 0$, $l \in \mathbb{N}$.
 Furthermore, observe that any 
  $x^m\nu \in \Ve(\mathbb{T}^1 \times X)$, where $\nu \in \Ve(X)$ and $m \in \mathbb{Z}$, can also be represented as
a single bracket:  
\begin{equation}\label{torustimes2} [x^l\d_x,\frac{1}{m-l+1}x^{m-l+1}\nu] =  \frac{1}{m-l+1}( x^l\d_x(x^{m-l+1})\nu
+ x^{m-l+1}  [x^l\d_x,\nu]) = x^{m}\nu,
\end{equation}
where $l \in  \mathbb{N}$, $m\in \mathbb{Z}$, and
$m-l+1 \neq 0$.

Let $\d \in \Ve(\mathbb{T}^1 \times X)$.  It can be represented as a finite sum of  
elements of the form $x^m\nu$, where $\nu \in \Ve(X)$, $m\in \mathbb{Z}$, and elements of the form  $x^mf\d_x$, where $f \in \mathcal{O}(X)$ and $m\in \mathbb{Z}$. Define $S\subset \mathbb N$ as the set of all exponents $m$ contained in all summands of $\d$. Now, if $l-1\notin S$ and $2l-1\notin S$, then 
 by \eqref{torustimes} and \eqref{torustimes2} we can represent $\d$ as a bracket $
 [x^l\d_x,\delta]$
for a suitable $\delta \in \Ve(\mathbb{T}^1 \times X)$. Since we have infinitely many  choices for $l \in \mathbb{N}$, we conclude that every $\d \in \Ve(\mathbb{T}^1 \times X)$ can be represented as a single bracket. 
\end{proof}

\begin{corollary} \label{cor1}
Let $G$ be a nontrivial connected linear algebraic $k$-group. Suppose that $G$ is not semisimple. Then $\Ve (G)$ is a simple Lie algebra of width one. 
\end{corollary}

\begin{proof}
For any connected linear algebraic $k$-group $G$, let $G^{\textrm{u}}$ denote its unipotent radical, then the quotient $G^{\textrm{red}}=G/G^{\textrm{u}}$ is reductive. The derived subgroup $G^{\textrm{ss}}=[G^{\textrm{red}},G^{\textrm{red}}]$ is semisimple, and the quotient $G^{\textrm{tor}}=G^{\textrm{red}}/G^{\textrm{ss}}$ is a torus, actually, the maximal toric quotient of $G$. The proof of the corollary is based on the following lemma which is of some interest in its own right. 

\begin{lemma} \label{lem:dec}
With the notation as above, we have an isomorphism of $k$-varieties
\begin{equation} \label{eq:lemma} 
G\cong G^{\textrm{u}}\times G^{\textrm{ss}} \times G^{\textrm{tor}}.     
\end{equation}
\end{lemma}

\noindent{\it {Proof of the lemma.}} We use a variation of the argument of \name{Popov} in the proof of Theorem 1  of \cite{Po2}.  
Denote by $G^{\textrm{ssu}}$ the kernel 
of the composed homomorphism $G\to G^{\textrm{red}}\to G^{\textrm{tor}}$. The group $G^{\textrm{ssu}}$, which is an extension of $G^{\textrm{ss}}$ by $G^{\textrm{u}}$, is generated by all unipotent subgroups of $G$ and is equal to the intersection of the kernels of all characters of $G$. 

Let us view the quotient morphism 
\begin{equation} \label{torsor} 
G\to G/G^{\textrm{ssu}}=G^{\textrm{tor}}
\end{equation}
as a $G^{\textrm{tor}}$-torsor $G\to G^{\textrm{tor}}$ under 
$G^{\textrm{ssu}}$ and show that it is isomorphic to the trivial torsor $(G/G^{\textrm{ssu}})\times G^{\textrm{ssu}}\to G^{\textrm{tor}}.$ To this end, choose a maximal torus $T\subset G$ and restrict the morphism \eqref{torsor} to $T$. 
We obtain a surjective morphism $T \to G^{\textrm{tor}}$, which admits a splitting $s\colon G^{\textrm{tor}} \to T$,
because the category of $k$-tori, being dual to the category of finitely generated free abelian groups, is semisimple.
The morphism $s$ induces a section of the morphism \eqref{torsor}, which splits the torsor $G\to  G^{\textrm{tor}}$. We thus obtain an isomorphism of $k$-varieties 
\begin{equation} \label{split1}
G\cong G^{\textrm{ssu}} \times G^{\textrm{tor}}. 
\end{equation}

In a similar manner, consider the quotient morphism $G^{\textrm{ssu}}\to G^{\textrm{ss}}$ and view it as a
$G^{\textrm{ss}}$-torsor under $G^{\textrm{u}}$.
As the group $G^{\textrm{u}}$ is unipotent, this torsor splits, see, e.g., a modern proof of Mostow's theorem 
on the Levi decomposition given by \name{Conrad} in  
\cite[proof~of~Proposition~5.4.1]{Con}.   We thus
obtain an isomorphism of $k$-varieties
\begin{equation} \label{split2}
G^{\textrm{ssu}} \cong G^{\textrm{u}} \times G^{\textrm{ss}}.
\end{equation}
Combining \eqref{split1} and \eqref{split2}, we obtain \eqref{eq:lemma}, which proves the lemma. \qed

We can now prove Corollary \ref{cor1}. Since by assumption the group $G$ is not semisimple, at least one of the groups $G^{\textrm{u}}$ and $G^{\textrm{tor}}$ is not trivial. By Lemma \ref{lem:dec}, the $k$-variety $G$ contains a direct factor of the form $\mathbb G_{\textrm{a}}^n$ or $\mathbb G_{\textrm{m}}^n$ with $n\ge 1$. The assertion of the corollary then follows from Propositions \ref{A1timesX} and  \ref{A*times}.

\end{proof}

We now consider a class of Lie algebras of vector fields on smooth affine varieties $X$ of dimension $n$ endowed with an algebraic volume form, that is, a nowhere vanishing global section $\omega$ of $\Lambda^n\Omega_{X/k}$. Note that $\omega$ is a closed algebraic differential form and that $H^0(X,\Lambda^n\Omega_{X/k})=\O(X)\omega$. Recall that the \emph{$\omega$-divergence} of an element $\mu$ of $\Ve(X)$ is the unique element $\Div_\omega \mu \in \mathcal{O}(X)$ such that $\mathcal{L}_\mu \omega=(\Div_\omega \mu)\omega$, where $\mathcal{L}_\mu \omega$ is the Lie derivative of $\omega $ with respect to $\mu$, which, in this case, is equal to the composition 
of the interior contraction
$i_\mu:\Lambda^n\Omega_{X/k}\to \Lambda^{n-1}\Omega_{X/k}$ with the exterior derivative $d\colon\Lambda^{n-1}\Omega_{X/k}\to \Lambda^{n}\Omega_{X/k}$.  An element $\mu \in \Ve(X)$ such that $\Div_\omega (\mu)=0$ is said to be of trivial $\omega$-divergence. The identity 
  \begin{equation}\label{Div}
\Div_\omega[\mu,\nu]=d(\Div_\omega \nu) \wedge i_\mu\omega- d(\Div_\omega \mu)\wedge i_\nu \omega=\mu (\Div_\omega \nu)-\nu(\Div_\omega \mu)
\end{equation}
which holds for all  $\mu,\nu \in \Ve(X)$ (see, e.g., \cite[Lemma~1]{Sh81}) implies in particular that the $k$-vector subspace $\mathrm{VP}_\omega(X)$ of $\Ve(X)$ generated by vector fields of trivial $\omega$-divergence is a Lie subalgebra of $\Ve(X)$. 

\begin{example}
For $X=\mathbb{A}^n$, $\omega_{\mathbb{A}^n}=dx_1\wedge \dots dx_n$ and an element $\mu = f_1\d_{x_1} + \dots+ f_n\d_{x_n}$ of  $\Ve(\mathbb{A}^n)$, we have the classical formula $$\Div_{\omega_{\mathbb{A}^n}} \mu = \frac{\d f_1}{\d x_1} + \dots + \frac{\d f_n}{\d x_n}.$$
By \cite[Lemma 3]{Sh81}, the Lie algebra $\mathrm{VP}_{\omega_{\mathbb{A}^n}}(\mathbb{A}^n)$ is simple. Note that according to \cite{Sh81}, this algebra coincides with the  Lie algebra of the so-called ind-group 

\[\SAut(\mathbb{A}^n) = \{ f = (f_1,\dots,f_n) \in \Aut(\mathbb{A}^n) \mid (\det  \Jac)(f) =\det \left[\frac{\d f_i}{\d x_j}\right]_{ij}= 1 \}.
  \] 
\end{example}

Let $X$ and $Y$ be smooth affine varieties with algebraic volumes $\omega_X$ and $\omega_Y$ respectively. Then $p_1^*\omega_{X}\wedge p_2^*{\omega_Y}$ is an algebraic volume form on $X\times Y$, which we denote for brevity by $\omega_{X}\wedge \omega_{Y}$. 
We have the following analogue of Proposition \ref{A1timesX} which implies in particular that for every $n\geq 1$ the width of the Lie algebra  $\mathrm{VP}_{\omega_{\mathbb{A}^n}}(\mathbb{A}^n)$ is equal to one.

\begin{proposition}\label{A1+affinespace}\label{affinespace}
Let $X$ be a smooth affine variety with algebraic volume form $\omega$. Assume that the $k$-linear map $\Div_\omega\colon\Ve(X)\to \mathcal{O}(X)$ is surjective.  Then for every $n\geq 1$  the  Lie algebra $\mathrm{VP}_{\omega_{\mathbb{A}^n}\wedge\omega}(\mathbb{A}^n\times X)$ has width equal to one. 
\end{proposition}
\begin{proof}
 Set $\mathbb{A}^1=\mathrm{Spec}(k[x])$ so that $\omega_{\mathbb{A}^1}=dx$. We first observe that for every smooth affine variety $Z$ with algebraic volume form $\omega_Z$ such that $\Div_{\omega_Z}\colon\Ve(Z)\to \mathcal{O}(Z)$ is surjective, the $k$-linear map $\Div_{dx\wedge \omega_Z}\colon\Ve(\mathbb{A}^1\times Z)\to \mathcal{O}(Z)[x]$ is surjective. 
 Indeed, every element of $\O(Z)[x]$ is a linear combination of elements $x^mf$ where $m\geq 0$ and $f\in \O(Z)$, and the hypothesis implies that $x^mf=\Div_{dx\wedge \omega_Z}(x^m\nu)$ where $\nu$ is any element of $\Ve(Z)$ such that $\Div_{\omega_Z}(\nu)=f$. 
By induction, using the previous observation, we are thus reduced to prove the assertion for $\mathbb{A}^1\times X$. Let $\mu$ be any element of $\mathrm{VP}_{dx\wedge \omega}(\mathbb{A}^1\times X)$. By the proof of Proposition \ref{A1timesX}, there exists $\delta \in \Ve(\mathbb{A}^1\times X)$ such that $\mu=[\partial_x,\delta]$. By \eqref{Div},  we have $$0 =  \Div_{dx\wedge \omega} \mu=\partial_x (\Div_{dx\wedge \omega}\delta)-\delta(\Div_{dx\wedge \omega}\partial_x).$$
Since $d\omega=0$, it follows that $\Div_{dx\wedge \omega}\partial_x=0$. This implies in turn that $\partial_x (\Div_{dx\wedge \omega}\delta)=0$ and hence, that $\Div_{dx\wedge \omega}\delta\in \mathcal{O}(Z)\subset \mathcal{O}(Z)[x]$. By assumption there exists $\nu\in \Ve(Z)$ such that $\Div_\omega(\nu)=-\Div_{dx\wedge \omega}\delta$. Then $\eta=\delta+\mathrm{p}_2^*\nu$ is an element of $\mathrm{VP}_{dx\wedge \omega}(\mathbb{A}^1\times X)$ such that $\mu=[\partial_x,\eta]$. Thus, 
$\mathrm{VP}_{dx\wedge \omega}(\mathbb{A}^1\times Z)=[\mathrm{VP}_{dx\wedge \omega}(\mathbb{A}^1\times Z),\mathrm{VP}_{dx\wedge \omega}(\mathbb{A}^1\times Z)]$ 
and it has bracket width equal to one. 
\end{proof}
\begin{remark} For a smooth affine variety $X$ of dimension $n$ with algebraic volume form $\omega$, the surjectivity of the map $\Div_\omega\colon\Ve(X)\to \mathcal{O}(X)$ implies in particular that every algebraic $n$-form on $X$ is exact, as follows from the definition  $d(i_\xi\omega)=(\Div_\omega \xi) \omega$. 
\end{remark}

\section{Vector fields on smooth affine curves} \label{sec:curves} 

In this section, we consider Lie algebras of algebraic vector fields on certain smooth affine curves. We begin with the case of rational curves. Recall that every such curve $C$ is isomorphic to a principal Zariski open subset of the affine line $\mathbb{A}^1=\mathrm{Spec}(k[x])$. 

\begin{proposition}
The width of the simple Lie algebra $\Ve (C)$ of a smooth affine rational curve $C$ is at most two. 
\end{proposition}
\begin{proof}
If $C=\mathbb{A}^1$, then the assertion follows from Proposition \ref{A1timesX}. We now assume that $C=\mathbb{A}^1\setminus\{p_1,\ldots, p_n\}$ for some $n\geq 1$. The $k$-vector space $\mathcal{O}(C)\simeq k[x,(\prod_{i=1}^{n}(x-p_i))^{-1}]$ has a basis
  $\{ x^i, \frac{1}{(x-p_1)^{j_1}}, \dots \frac{1}{(x-p_n)^{j_n}}  \}$, $i\geq 0$, $j_\ell \geq 1$. 
  Note that $[\d_x, x^i\d_x] = i x^{i-1} \d_x$ and that
 \[
 [\d_x, \frac{1}{(x-p_r)^{j_r}}\d_x]
 =  \frac{-j_r}{(x-p_r)^{j_r+1}}\d_x. 
 \]
 Therefore, any element of the form $P\d_x$, where $P$ does not contain elements proportional to one of the elements $\frac{1}{(x-p_i)}$, $i=1,\ldots n$, as a summand can be represented as a bracket of the element $\d_x$ and some suitable element from $\Ve(C)$.
  Since  on the other hand
 \[
 [x\d_x, \frac{1}{x-p_i}\d_x]= - \frac{2}{x-p_i}\d_x - \frac{p_i}{(x-p_i)^2}\d_x,
 \]
 it follows that that every $\mu \in \Ve(C)$ can be written in the form $\mu = [\d_x,\nu] + [x\d_x,\delta]$ for some suitable $\nu, \delta \in \Ve(C)$.
\end{proof}

In contrast with the case of the $1$-dimensional torus $\mathbb{T}^1\simeq \mathbb{A}^1\setminus \{0\}$ for which the vector field $x^{-1}\partial_x$ can be written as the single Lie bracket $[x^{m}\partial_x,-\frac{1}{2m}x^{-m}\partial_x]$ for any $m\neq 0$, we expect that for every  $n\geq 2$, a vector field on $\mathbb{A}^1\setminus \{p_1,\dots, p_n\}$ with single pole at each of the points $p_1,\dots p_n$ cannot be written as the Lie bracket of two elements of $\Ve(\mathbb{A}^1\setminus \{p_1,\dots, p_n\})$. This motivates the following:
 \begin{conjecture}
The width of the Lie algebra $\Ve(\mathbb{A}^1 \setminus \{ p_1,\dots,p_n\})$, $n\geq 2$, equals two. 
 \end{conjecture}

 We now consider non-rational smooth affine curves.  Recall that an irreducible smooth affine curve $C$ is said to have \emph{a unique place at infinity} if it is equal to the complement of a single closed point in a smooth projective curve $\bar{C}$. 

\begin{theorem}\label{polefiltration} \label{prop:curve-higher-width} \label{main2} Let $C$ be an irreducible smooth affine curve with a unique place at infinity. Assume that $C$ is not rational and that the tangent sheaf $\mathcal{T}_C$ of $C$ is trivial. Then the width of the simple Lie algebra $\mathrm{Vec}(C)$ is larger than or equal to two. 

More precisely, no nowhere vanishing vector field on $C$ can  be equal to the Lie bracket of two other vector fields on $C$.
\end{theorem}

\begin{proof}
Let $\bar{C}$ be the smooth projective model of $C$, and let $c_\infty=\bar{C}\setminus C$.  By assumption, the genus $g$ of $\bar{C}$ positive. Note that $\O(C)^*=k^*$. Indeed, otherwise, there would exist a dominant morphism $f\colon C\to \mathbb{A}^1\setminus \{0\}$. The latter would extend to a surjective morphism $\bar{f}\colon\bar{C}\to \mathbb{P}^1$ with the property that $\bar{f}^{-1}([0:1])$ and $\bar{f}^{-1}([1:0])$ belong to $\bar{C}\setminus C$, which is impossible since $C$ has a unique place $c_\infty$ at infinity. 

Since by hypothesis $\mathcal{T}_C$ is the trivial sheaf $\O_C$, we have $\mathrm{Vec}(C)=\O(C)\cdot \tau$ for a certain nowhere vanishing global vector field $\tau\in \mathrm{Vec}(C)$, unique up to multiplication by a nonzero constant. Since $\deg(\mathcal{T}_{\bar{C}})=2-2g$, we have $\mathcal{T}_{\bar{C}}=\mathcal{O}_{\bar{C}}((2-2g)c_\infty)$.  

Now suppose that $\tau$ can be written in the form $[\xi, \nu]$
for some $\xi,\nu \in \mathrm{Vec}(C)$. Write $\xi=f\tau$ and $\nu=g\tau$ for some non-zero regular functions $f$ and $g$ on $C$. Viewing $f$ and $g$ as rational functions on $\bar{C}$, we can assume further without loss of generality that $\mathrm{ord}_{c_\infty}(f)\neq \mathrm{ord}_{c_\infty}(g)$. Indeed, let $R=\mathcal{O}_{\bar{C},c_\infty}$ be the local ring of  $\bar{C}$ at $c_\infty$, and let $t$ be a uniformizing parameter in its maximal ideal. If the equality holds, say $\mathrm{ord}_{c_\infty}(f)=\mathrm{ord}_{c_\infty}(g)=\ell$, then the classes of $f$ and $g$ in $R[t^{-1}]=\mathrm{Frac}(R)$ are equal to $a_ft^\ell$ and $a_gt^\ell$ for some uniquely determined element $a_f$ and $a_g$ in $R\setminus tR$. Let $\overline{a_f}$ and $\overline{a_g}$ be the residue classes of $a_f$ and $a_g$ in $R/tR=k$. Then $\overline{a_f}, \overline{a_g}\in k^*$. Let $\lambda  =\overline{a_f}/\overline{a_g}$ so that $a_f-\lambda a_g\in tR$. Then $\mathrm{ord}_{c_\infty}(f-\lambda g)\geq \ell+1$. On the other hand, since $\lambda \in k^*$, we have $[(f-\lambda g)\tau,g\tau]=[f\tau,g\tau]=\tau.$

The vector field $\tau$ has a pole of order $2g-2$ at $\infty$. Write $n_\xi=\mathrm{ord}_{c_\infty} \xi=2-2g+\mathrm{ord}_{c_\infty}(f)$ and $n_\nu=\mathrm{ord}_{c_\infty} \nu=2-2g+\mathrm{ord}_{c_\infty}(g)$. By construction, we have $n_\xi\neq n_\nu$. 
It follows that  
\begin{equation}\label{main-ord-equality}
    2-2g=\mathrm{ord}_{c_\infty}(\tau)=\mathrm{ord}_{c_\infty}([\xi,\nu])=n_\xi+n_\nu -1=2(2-2g)+\mathrm{ord}_{c_\infty}(f)+\mathrm{ord}_{c_\infty}(g)-1.
\end{equation}
Since $g\geq 1$ and $\mathrm{ord}_{c_\infty}(f)$ and $\mathrm{ord}_{c_\infty}(f)$ are both non-positive, this is impossible.
\end{proof}

\begin{corollary} \label{cor:nonrat-plane-curve} Let $C\subset \mathbb{A}^2$ be an irreducible non-rational smooth affine curve with a unique place at infinity. Then the width of $\mathrm{Vec}(C)$ is larger than or equal to two.
\end{corollary}
\begin{proof}
Indeed, since $C$ is a smooth plane curve, we have $\mathcal{T}_C=\Lambda^2 \mathcal{T}_{\mathbb{A}^2}|_C \otimes \mathcal{C}_{C/\mathbb{A}^2} $, where  $\mathcal{C}_{C/\mathbb{A}^2}$ is the conormal sheaf of $C$ in $\mathbb{A}^2$. Since $\mathbb{A}^2$ is factorial, the ideal sheaf  $\mathcal{I}_C\subset \mathcal{O}_{\mathbb{A}^2} $
of $C$ is isomorphic to the trivial invertible sheaf $\mathcal{O}_{\mathbb{A}^2}$, and so $\mathcal{C}_{C/\mathbb{A}^2}=\mathcal{I}_C/\mathcal{I}_C^2$ is isomorphic to $\mathcal{O}_C$. Since $\Lambda^2 \mathcal{T}_{\mathbb{A}^2}\simeq \mathcal{O}_{\mathbb{A}^2}$ is the trivial invertible sheaf as well, we conclude that $\mathcal{T}_C$ is trivial. The assertion of the corollary now follows from Theorem \ref{polefiltration}. 
\end{proof}

A simple class of curves satisfying the hypotheses of Theorem $\ref{main2}$ is the following family of affine hyperelliptic curves, borrowed from the paper of \name{Billig} and \name{Futorny} \cite{BF} where many properties of their Lie algebras of vector fields algebras had been established.

\begin{example}\label{example-hyper-elliptic} Let $g\geq 1$,  and let $C\subset \mathbb{A}^2=\mathrm{Spec}(\mathbb{C}[x,y])$ be the smooth affine curve defined by the equation $y^2=h(x)$ for some separable  polynomial  $h(x)$ of degree $2g+1$.  The smooth projective model $\bar{C}$ of $C$ is a hyperelliptic curve on which $\mathrm{pr}_x\colon C\to \mathbb{A}^1$ extends to a double cover $\pi\colon\bar{C}\to \mathbb{P}^1$. Since $\mathrm{deg}(h)=2g+1$ is odd, it follows from the Riemann--Hurwitz formula that $\bar{C}$ has genus $g$ and that $\infty=\mathbb{P}^1\setminus \mathbb{A}^1$ is a branch point of $\pi$. So, $\bar{C}\setminus C=\pi^{-1}(\infty)$ consists of a unique point, which shows that $C$ has a unique place at infinity. By Corollary \ref{cor:nonrat-plane-curve}, the width of the Lie algebra $\mathrm{Vec}(C)$ is larger than or equal to two. 

We note that in \cite[Section 5]{BF}, the considered filtrations on $\mathcal{O}(C)$ and $\Ve(C)$ coincide with those induced by the pole order at the unique place at infinity of $C$. 
\end{example}

Recall that an element $\eta$ in the Lie algebra $\Ve(X)$ of vector fields on a smooth affine variety $X$ is called {\it semisimple} if the map $\ad(\eta)\colon\Ve(X)\to \Ve(X)$, $\xi \mapsto [\eta,\xi]$, has an eigenvector, and that $\eta$ is called {\it nilpotent} if there is a nonzero vector annihilated by $\ad(\eta)$. The following corollary generalizes \cite[Theorem 5.3]{BF} which corresponds to the special curves considered in Example \ref{example-hyper-elliptic}.

\begin{corollary} \label{no-ss}
Let $C$ be an irreducible smooth affine curve with a unique place at infinity. Assume that $C$ is not rational and that the tangent sheaf $\mathcal{T}_C$ of $C$ is trivial. Then 
\begin{enumerate}
\item  $\Ker \ad (\xi) = k \xi$;
\item $\Ve(C)$  has no non-trivial semisimple  elements;
 \item $\xi\not\in \Imm \ad \xi$;
\item $\Ve(C)$  has no non-trivial  nilpotent elements.
\end{enumerate}
\end{corollary}
\begin{proof}
It is clear that $\xi \in \Ker(\xi)$. To prove $(1)$ assume that $[\xi,\eta] = 0$.   Write $\xi=f\tau$ and $\nu=g\tau$ for some non-zero regular functions $f$ and $g$ on $C$. From \eqref{main-ord-equality}
we have $\mathrm{ord}_{c_\infty}([\xi,\nu])=n_\xi+n_\nu -1$
which implies that
\begin{equation}\label{one-dim-Ker}
   n_\xi+n_\nu -1 =\mathrm{ord}_{c_\infty} \xi + \mathrm{ord}_{c_\infty} \nu - 1=2(2-2g)+\mathrm{ord}_{c_\infty}(f) +\mathrm{ord}_{c_\infty}(g) -1= \mathrm{ord}_{c_\infty}(0) =0
\end{equation} 
which is not possible as
 $g \ge 1$, $\mathrm{ord}_{c_\infty}(f)\le 0$, $\mathrm{ord}_{c_\infty}(g) \le 0$.

(2) We claim that the relation $[\xi, \nu] = \lambda \nu$ with $\lambda \in k$, $\lambda \neq 0$, cannot hold for non-zero elements in $\Ve(C)$. 
Indeed, otherwise, from \eqref{main-ord-equality}
we have $\mathrm{ord}_{c_\infty}([\xi,\nu])=n_\xi+n_\nu -1$
which implies that
\begin{equation}\label{no-semisimple-elements}
    n_\xi =\mathrm{ord}_{c_\infty} \xi=2-2g+\mathrm{ord}_{c_\infty}(f)=1.
\end{equation} Since $g \ge 1$ and $\mathrm{ord}_{c_\infty}(f) \le 0$, equation \eqref{no-semisimple-elements} does not hold.

(3) If $\xi\in \Imm \ad \xi$, then there is $\eta \in \Ve(C)$ such that $[\xi,\eta] = \xi$,  
but this is not possible by (2).

(4) If $\ad \xi$ is nilpotent, then there exists a non-zero $\nu \in \Imm \ad \xi$ such that $[\xi, \nu] = 0$.
By part (1), $\nu$ must then be a multiple of $\xi$, which  contradicts (3).
\end{proof}

\begin{example} \label{ex:quartic-tangent} Smooth non-hyperelliptic affine plane curves satisfying the hypotheses of Corollary \ref{cor:nonrat-plane-curve} are obtained from  smooth projective quartics $\bar{C}\subset \mathbb{P}^2$ admitting a $4$-tangent $L$, that is, a tangent line $L$ to $\bar{C}$ intersecting $\bar{C}$ at a single point $c_\infty$. Recall that a smooth quartic curve $\bar{C} \subset \mathbb{P}^2$ has genus $3$ and is not hyperelliptic. Given any $4$-tangent line $L$ to $\bar{C}$, the affine plane curve $C=\bar{C}\setminus L \subset \mathbb{P}^2\setminus L=\mathbb{A}^2$ has a unique place at infinity and is non-rational and non-hyperelliptic. An explicit example is given by the smooth quartic $\bar{C}\subset \mathbb{P}^2$ with equation $x^4+z(x^3+y^3+z^3)=0$ and the $4$-tangent line $L=\{z=0\}$, with associated affine plane curve $C=\{x^4+x^3+y^3+1=0\}\subset \mathbb{A}^2$.
\end{example}

\begin{remark} More generally, pairs $(\bar{C},c_\infty)$ consisting of a smooth projective non-hyperelliptic curve $\bar{C}$ and a point $c_\infty$ on $\bar{C}$ such that $C:=\bar{C}\setminus\{c_\infty\}$ satisfies the hypotheses of Theorem   \ref{main2} can be characterized as follows. Since $\bar{C}$ is not hyperelliptic, we have $g=g(\bar{C})\geq 3$,  and the canonical morphism $\bar{C}\to  \mathbb{P}^{g-1}=\mathbb{P}(H^0(\bar{C},\Omega_{\bar{C}/k}))$ is a closed embedding. By definition, every hyperplane section of $\bar{C}$ is a canonical divisor on $\bar{C}$. It follows that for every hyperplane $H$ of  $\mathbb{P}^{g-1}$ such that $\bar{C}\cap H$ consists of a unique point $c_\infty$, $C=\bar{C}\setminus\{c_\infty\}$ is a smooth affine curve with a unique place at infinity whose canonical sheaf (hence also  tangent sheaf) is trivial. For $g=3$, we recover the smooth plane quartics $\bar{C}\subset \mathbb{P}^2$ with a $4$-tangent line of Example \ref{ex:quartic-tangent}.
\end{remark}

%%%%%%%%%%%%%%%%%%%%%%%%%%%%%%%
\section{Lie algebras of symplectic vector fields on affine surfaces}\label{section3.1}
In this section, we consider certain natural Lie subalgebras of vector fields on smooth affine surfaces endowed with an algebraic volume form. 
%In the complex case, such surfaces have been intensively %studied in particular by \name{Kaliman} and %\name{Kutzschebauch} in connection to the algebraic volume %density property, see \cite{KK10}, \cite{KK16} and the %references therein. 
We first collect basic notions and results about symplectic, Hamiltonian and locally nilpotent vector fields on symplectic affine surfaces. We then use this framework to proceed to the proof of Theorem B. 
%\adrien{Introduction to be reworked}
%This is primarily devoted to the proof of Theorem B.  Before we proceed to the proof itself, let us put the study of certain Lie algebras of vector fields on Danielewski surfaces into the broader context of the study of Lie subalgebras of symplectic algebraic vectors fields on symplectic affine surfaces \adrien{Here, as a motivation/justification for the naturality and "importance" of these notions, we should/have to say a word about the series of work by Kaliman-Kutzschebauch and al. around the algebraic volume density property and such things and add a couple of relevant references. We should also refer to [KL13] as a general reference for some detailed proofs, indicating that the results there are stated for $k=\mathbb{C}$ but that they carry on verbatim to arbitrary algebraically closed fields of characteristic zero}. 

\subsection{Hamiltonian algebraic vector fields on symplectic affine surfaces}

\begin{definition}
A \emph{symplectic affine surface} is a pair $(S,\omega)$ consisting of a smooth affine surface $S$ and a nowhere vanishing algebraic $2$-form $\omega\in H^0(S,\Lambda^2  \Omega_{S/k})$ on $S$.  

A \emph{symplectomorphism}
%(resp. \emph{anti-symplectomorphism})
between symplectic affine surfaces $(S,\omega)$ and $(S',\omega')$ is a morphism $\varphi\colon S\to S'$ such that $\varphi^*\omega'=\omega.$ 
%(resp. $\varphi^*\omega'=-\omega$).
An action of an algebraic group $G$ on $S$ is said to be symplectic if the image of the corresponding group homomorphism $G\to \mathrm{Aut}_k(S)$ consists of symplectomorphisms of $(S,\omega)$.
\end{definition}

For a symplectic affine surface $(S,\omega)$, we have  $H^0(S,\Lambda^2 \Omega_{S/k})=\O(S)\cdot \omega$.
Furthermore, since $\omega$ is in particular non-degenerate, the interior contraction induces an isomorphism 
\begin{equation}\label{contraction-iso}
\Psi_\omega\colon\Ve(S)\to H^{0}(S,\Omega_{S/k}),\quad \xi \mapsto i_{\xi}\omega.
\end{equation}

\subsubsection{Recollection on symplectic and Hamiltonian vector fields} 
Let $(S,\omega)$ be a symplectic affine surface. 
%Since $\omega$ is a nowhere vanishing, hence 
For a vector field $\xi \in\Ve(S)$, we denote by $\mathcal{L}_{\xi}$ the Lie derivative with respect to $\xi$. Recall that by Cartan's formula, for every algebraic $p$-form $\alpha$ on $S$ and every $\xi \in \Ve(S)$, we have $\mathcal{L}_{\xi}\alpha=d(i_{\xi}\alpha)+i_{\xi}(d\alpha)$ where $d$ is the de Rham differential on algebraic forms. Since $\omega$ is a closed form, we have simply $\mathcal{L}_{\xi}\omega=d(i_{\xi}\omega)=\Div_\omega(\xi)\omega$. 
\begin{definition}
A \emph{symplectic algebraic vector field} on a symplectic affine surface $(S,\omega)$ is an element $\xi$ of  $\Ve(S)$ such that
$\mathcal{L}_{\xi}\omega=0$, equivalently such that $i_{\xi}\omega\in H^{0}(S,\Omega_{S/k})$ is a closed $1$-form. A symplectic vector field $\xi$ on $S$ is called \emph{Hamiltonian} if the $1$-form $i_{\xi}\omega$ is exact. 
\end{definition}

For every  $\xi,\nu \in \Ve(S)$ , we have the identities  \begin{equation}\label{LbracketContraction}
\mathcal{L}_{[\xi,\nu]}\omega=\mathcal{L}_{\xi}\mathcal{L}_{\nu}\omega-\mathcal{L}_{\nu}\mathcal{L}_{\xi}\omega\quad\textrm{and}\quad i_{[\xi,\nu]}\omega=\mathcal{L}_{\xi}i_{\nu}\omega-i_{\nu}\mathcal{L}_{\xi}\omega=i_{\xi}\mathcal{L}_{\nu}\omega-\mathcal{L}_{\nu}i_{\xi}\omega.
\end{equation}
These imply that the $k$-vector subspace $\mathrm{VP}_\omega(S)\subseteq\mathrm{Vec}(S)$ of symplectic vector fields is a Lie subalgebra and that the $k$-vector subspace $\mathcal{H}_\omega(S)\subseteq \mathrm{VP}_\omega(S)$ of Hamiltonian vector fields is a Lie subalgebra. The second identity implies in addition that  $\mathcal{H}_\omega(S)$ contains  $[\mathrm{VP}_\omega(S),\mathrm{VP}_\omega(S)]$, hence is a Lie ideal of $\mathrm{VP}_\omega(S)$. In fact, $\mathcal{H}_\omega(S)$ is equal to the kernel of the $k$-linear map $$\mathrm{VP}_\omega(S)\to \mathrm{H}^1_{\mathrm{dR}}(S),\; \xi \mapsto \overline{i_\xi \omega}, $$
where $\mathrm{H}^*_{\mathrm{dR}}(S)$ denotes the cohomology of the naive algebraic de Rham complex $$0\to \mathcal{O}(S)\stackrel{d}{\to} H^0(S,\Omega_{S/k}) \stackrel{d}{\to} H^0(S,\Lambda^2\Omega_{S/k})\to 0$$
of the smooth affine surface $S$, see \cite{Gro66}.

\medskip 

For every $f\in \mathcal{O}(S)$, denote by $\theta_{f}\in \mathcal{H}_\omega(S)\subset \Ve(S)$ the unique vector field such that $\Psi_\omega(\theta_{f})=i_{\theta_{f}}\omega=df$ (see \eqref{contraction-iso}). 
For every two elements $f,g\in \mathcal{O}(S)$, it follows from \eqref{LbracketContraction} that $i_{[\theta_{f},\theta_{g}]}\omega=\mathcal{L}_{\theta_f}i_{\theta_g}\omega=d(\omega(\theta_{g},\theta_{f}))$. Since $\Psi$ is an isomorphism, we have  $[\theta_{f},\theta_{g}]=\theta_{\omega(\theta_{g},\theta_{f})}$, and we define the \emph{Poisson bracket} of $f$ and $g$ as the element $\{f,g\}_\omega=\omega(\theta_g,\theta_f)$  of $\mathcal{O}(S)$. 
 Equivalently, $\{f,g\}_\omega$ is the unique element of $\O(S)$ such that 
 \begin{equation}\label{PoissonB}
 df\wedge dg=(i_{\theta_f}\omega )\wedge (i_{\theta_g} \omega)=(\{f,g\}_\omega)\omega. 
 \end{equation}
 For every $f\in \O(S)$, the $k$-linear homomorphism $\{f,\cdot\}_\omega:\O(S)\to \O(S)$ is a $k$-derivation equal to the Lie derivative $\mathcal{L}_{\theta_f}=i_{\theta_f}\circ d$.
 The bracket $\{\cdot,\cdot\}_\omega$ defines a structure of Lie algebra on $\mathcal{O}(S)$ for which the surjective $k$-linear homomorphism 
\begin{equation}\label{PoissonIso}
\Theta_\omega\colon\O(S) \to \mathcal{H}_\omega(S), \; f\mapsto \theta_f
\end{equation} is a homomorphism of Lie algebras with kernel $k\subset \O(S)$, inducing an isomorphism 
\begin{equation}\label{IsoLie}
(\mathcal{O}(S)/k,\{\cdot,\cdot\}_\omega)\cong(\mathcal{H}_\omega(S),[\cdot,\cdot]).
\end{equation}

\medskip 

Let $(S,\omega)$ be a symplectic affine surface. Denote by  $E_\omega(S) \subset \O(S)$ the 
image of the $k$-linear map $\mathrm{Div}_\omega\colon\Ve(S)\to \O(S)$. Since every algebraic $1$-form $\alpha$ on $S$ is equal to $i_\xi\omega$ for some uniquely determined element $\xi=\Psi_\omega^{-1}(\alpha)$ of $\Ve(S)$, the identity $\mathcal{L}_\xi\omega=d(i_\xi\omega)=\mathrm{Div}_\omega(\xi)\omega$ implies that $E_\omega(S)$ consists precisely of the elements $h\in \O(S)$ such that the $2$-form $h\omega\in H^0(S,\Lambda^2\Omega_{S/k})$ is exact. It follows in turn from \eqref{PoissonB} that $E_\omega(S)$ is a Lie ideal of
 $(\mathcal{O}(S),\{\cdot,\cdot\}_\omega)$ containing $\{\mathcal{O}(S),\mathcal{O}(S)\}_\omega$ and that the quotient $\mathcal{O}(S)/E_\omega$ is isomorphic to 
$\mathrm{H}^2_{\mathrm{dR}}(S)$. The image of $E_\omega(S)$  under the surjective homomorphism \eqref{PoissonIso} is then a Lie ideal $\mathcal{E}_\omega(S)$ of $\mathcal{H}_\omega(S)$ of finite codimension. This ideal contains $[\mathcal{H}_\omega(S),\mathcal{H}_\omega(S)]$.

\subsubsection{Locally nilpotent vector fields and additive group actions}
Recall that a $k$-derivation  $\partial$ of the coordinate ring of an affine algebraic variety $X$ is called \emph{locally nilpotent} if for every $f \in \O(X)$ there exists $s \in \mathbb{N}$ such that $\partial^sf = 0$. There is a well-known one-to-one correspondence between algebraic actions of the additive group  $\mathbb{G}_a=(\mathrm{Spec}(k[t]),+)$ on $X$
and locally nilpotent $k$-derivations of $\O(X)$. 
It is given by associating to an action $\sigma\colon\mathbb{G}_a\times X\to X$ the $k$-derivation $$\partial_\sigma=\frac{d}{dt}|_{t=0}\circ\sigma^{*}\in\mathrm{Der}_{k}(\O(X),\O(X)),$$ 
which, under the identification $\mathrm{Der}_{k}(\O(X),\O(X))=\Ve(X)$, coincides with the tangent vector field to the orbits of $\sigma$. 
\begin{lemma}\label{lem:GaSymp}Every algebraic $\mathbb{G}_a$-action on a symplectic affine surface $(S,\omega)$ is symplectic. 
\end{lemma}
\begin{proof} 
Let $\sigma\colon\mathbb{G}_a\times S\to S$ be a non-trivial $\mathbb{G}_a$-action. For every $t\in k$, there exists a unique element $f_t\in \O(S)^*$ such $\sigma(t,\cdot)^*\omega=f_t\omega$. Since there is no non-constant morphism from $\mathbb{G}_a$ to $\mathbb{A}^1\setminus \{0\}$, there exists a unique element $F\in \O(S)^*$ such that the morphism 
$$ \alpha\colon\mathbb{G}_a\times S\to \mathbb{A}^1\setminus \{0\}, \quad (t,x)\mapsto f_t(x)$$
factors as $F\circ \mathrm{p}_2$. 
Since $\alpha(t+t',x)=\alpha(t',\sigma(t,x))\cdot  \alpha(t,x)$, we have $F(x)=F(\sigma(t,x))F(x)$ for all $t\in k$ and $x\in S$. The fact that a $\mathbb{G}_a$-orbit consists either of a fixed point or of a curve isomorphic to $\mathbb{A}^1$ implies in turn that $F(x)=F(x)^2$ for all $x\in S$. Thus, $F$ is the constant function $1$ on $X$, and hence, $f_t=1$ for all $t\in k$. 
%Since $\mathbb{G}_{a,k}$ is connected and $\O(X)^*/\mathbb{G}_ {m,k}$ is a finitely generated abelian group \cite[Tag 04L3]{StP},
%The composition $\mathbb{G}_{a,k}\to \O(X)^*/\mathbb{G}_ {m,k}$ is the trivial homomorphism. Thus, $\alpha$ factors through a homomorphism $\mathbb{G}_{a,k}\to \mathbb{G}_{m,k}$,  which, in turn, coincides with that induced the constant homomorphism of groups schemes $1:\mathbb{G}_{a,k}\to \mathbb{G}_{m,k}$. 
\end{proof}

%{\color{red} This proof is too compressed for me. In %particular, I do not understand the group scheme structure on %$\O(X)^*$. Maybe a precise reference to Stack Project can %help. }
%\adrien{ Proof updated. I wrote an an alternative detailed  %argument based on more elementary considerations - The %previous more "abstract" argument, which was meant to work %independently on the base field (i.e. by avoiding the use of %points) did need some fixing anyway, so thanks for noticing %the problem.}

It follows from Lemma \ref{lem:GaSymp} that every locally nilpotent vector field $\partial$ on a symplectic affine surface $(S,\omega)$ is symplectic. A $\mathbb{G}_a$-action on $S$ is called \emph{Hamiltonian} if its associated locally nilpotent vector field $\partial$ is Hamiltonian. We denote by $\langle \mathrm{LNV}(S)\rangle \subseteq \mathrm{VP}_\omega(S)$ and $\langle \mathcal{H}_\omega\mathrm{LNV}(S)\rangle \subseteq \mathcal{H}_\omega(S)$ the Lie subalgebras generated respectively by locally nilpotent and Hamiltonian locally nilpotent vector fields on $(S,\omega)$. Summing up, for a symplectic affine surface $(S,\omega)$ we have the following diagram of Lie subalgebras of $\mathrm{VP}_\omega(S)$:
\begin{equation}
\begin{tikzcd}
 \langle \mathrm{LNV}(S) \rangle  \arrow[r]& \mathrm{VP}_\omega(S) &  \\
 \langle \mathcal{H}_\omega\mathrm{LNV}(S) \rangle \arrow[r] \arrow[u] & \mathcal{H}_\omega(S) \arrow[u] &  \left[\mathrm{VP}_\omega(S),\mathrm{VP}_\omega(S)\right] \arrow[l] \\
 & \mathcal{E}_\omega(S) \arrow[u] &  \left[\mathcal{H}_\omega(S),\mathcal{H}_\omega(S)\right]  \arrow[u] \arrow[l]
\end{tikzcd}
\end{equation}
 Symplectic affine surfaces thus provide a large class of surfaces for which the question of the simplicity and if so, of their bracket width, of the natural Lie algebras $\langle \mathcal{H}_\omega\mathrm{LNV(S)}\rangle$,  $\mathcal{E}_\omega(S)$ and  $[\mathcal{H}_\omega(S),\mathcal{H}_\omega(S)]$, as well as their intersections, can be further studied. 

\begin{example} \label{ex:sympA2}Let $S=\mathbb{A}^2=\mathrm{Spec}(k[x,y]$ and $\omega=\omega_{\mathbb{A}^2}=dx\wedge dy$. Then we have $$\mathrm{VP}_{\omega_{\mathbb{A}^2}}(\mathbb{A}^2)=\mathcal{H}_{\omega_{\mathbb{A}^2}}(\mathbb{A}^2)=\mathcal{E}_{\omega_{\mathbb{A}^2}}(\mathbb{A}^2)=[\mathcal{H}_{\omega_{\mathbb{A}^2}}(\mathbb{A}^2),\mathcal{H}_{\omega_{\mathbb{A}^2}}(\mathbb{A}^2)]=\lielnd{\mathbb{A}^2}$$
(see, e.g., \cite[Lemma~2]{Sh81}). Furthermore, this Lie algebra is simple and it has width equal to one by  Proposition \ref{affinespace}.
\end{example}

%\adrien{Note sure the next example is really interesting or informative ... I wanted to make a connection to the non-rational curve case of the previous section but on a surface which on the other hand has nontrivial $\mathbb{G}_a$-actions. }

\begin{example}\label{ex:nonratxA1} Let $C$ be a smooth non-rational affine curve with a nowhere vanishing algebraic $1$-form $\omega_C$.  The surface $S=C\times \mathbb{A}^1$ is endowed with the nowhere vanishing $2$-form $\omega=\omega_C\wedge dx$. Since $C$ is not rational, we have  $\mathrm{H}^1_{\mathrm{dR}(S)}\cong \mathrm{H}^1_{\mathrm{dR}(C)}\neq 0$. So $\mathcal{H}_\omega(S)$ is a proper Lie ideal of $\mathrm{VP}_\omega(S)$. Since $\mathrm{H}^2_{\mathrm{dR}}(S)=0$, it follows that $E_\omega(S)=\O(S)$ and hence, that $\mathcal{E}_\omega(S)=\mathcal{H}_\omega(S)$. Since $C$ is not rational, every $\mathbb{G}_a$-action on $S$ is given by a locally nilpotent $k$-derivation of $\O(S)=\O(C)[x]$ of the form $\delta_f=f\frac{\partial}{\partial x}$, where $f\in \O(C)$. The Lie algebra $\lielnd{S}$ is thus abelian. Furthermore, since the  $\mathbb{G}_a$-action determined by $\delta_f$ is Hamiltonian if and only if the $1$-form $f\omega_C$ on $C$ is exact, $\langle \mathcal{H}_\omega\mathrm{LNV}(S)\rangle$ is a proper subalgebra of $\lielnd{S}$. 
\end{example}

\begin{proposition} \label{ex:sympT2}For the symplectic torus   $\mathbb{T}^2=\mathrm{Spec}\, k[x^{\pm 1},y^{\pm 1}]$ endowed with the $2$-form $\omega_{\mathbb{T}^2}=\frac{dx}{x}\wedge \frac{dy}{y}$, the following hold:
\begin{enumerate} 
\item $\langle \mathcal{H}_{\omega_{\mathbb{T}^2}}\mathrm{LNV}(\mathbb{T}^2)\rangle=\langle \mathrm{LNV}(\mathbb{T}^2)\rangle=0$
\item $\mathcal{H}_{\omega_{\mathbb{T}^2}}(\mathbb{T}^2)$ is a proper Lie ideal of codimension $2$ of $\mathrm{VP}_{\omega_{\mathbb{T}^2}}(\mathbb{T}^2)$.
\item $\mathcal{E}_{\omega_{\mathbb{T}^2}}(\mathbb{T}^2)=[\mathcal{H}_{\omega_{\mathbb{T}^2}}(\mathbb{T}^2),\mathcal{H}_{\omega_{\mathbb{T}^2}}(\mathbb{T}^2)]=\mathcal{H}_{\omega_{\mathbb{T}^2}}(\mathbb{T}^2)$ is a simple Lie algebra of width one. 
\end{enumerate}
\end{proposition}
\begin{proof}
Assertion (1) is an immediate consequence of the fact that $\mathbb{T}^2$ does not admit any non-trivial $\mathbb{G}_a$-action.  Assertion (2) follows from the observation that $\mathrm{H}^1_{\mathrm{dR}}(\mathbb{T}^2)$ is isomorphic to $k^{\oplus 2}$, generated by the classes of the $1$-forms $$x^{-1}dx=-i_{y\tfrac{\partial}{\partial y}}\omega \quad \textrm{and} \quad  y^{-1}dy=i_{x\tfrac{\partial}{\partial x}}\omega.$$
On the other hand, the class of $\omega_{\mathbb{T}^2}$ generates $\mathrm{H}^2_{\mathrm{dR}}(\mathbb{T}^2)\cong k$. It follows that the subspace $E_{\omega_{\mathbb{T}^2}}\subset k[x^{\pm 1},y^{\pm 1}]$ consists of all elements with trivial constant term. Since $\mathcal{H}_{\omega_{\mathbb{T}^2}}(\mathbb{T}^2) \cong \O(\mathbb{T}^2)/k$, it follows that $\mathcal{H}_{\omega_{\mathbb{T}^2}}(\mathbb{T}^2)=\mathcal{E}_{\omega_{\mathbb{T}^2}}(\mathbb{T}^2)$. 

To complete the proof of assertion (3), it remains to show that $\mathcal{E}_{\omega_{\mathbb{T}^2}}(\mathbb{T}^2)$ is a simple Lie algebra of width one. 
The Poisson bracket on $\O(\mathbb{T}^2)$ associated to $\omega_{\mathbb{T}^2}$ is given by $\{x,y\}=xy$ and extension by linearity and the Jacobi identity. Thus, for $(k,\ell)$ and $(m,n)$ in  $\mathbb{Z}^2\setminus (0,0)$, we have 
\begin{equation}\label{bracketH_T}
\{x^ky^\ell,x^my^n\}=(kn - \ell m)x^{k+m}y^{\ell +n}.
\end{equation}
Since every element $f$ of $\O(\mathbb{T}^2)$ with trivial constant term is a linear combination of finitely many monomials of the form $x^{\alpha_i}y^{\beta_i}$ with $(\alpha_i,\beta_i)\in \mathbb{Z}^2\setminus (0,0)$, we can find $(k,l)\in \mathbb{Z}^2\setminus (0,0)$ such that $k(\beta_i-\ell)- \ell(\alpha_i-k)\neq 0$ for every $i$. Then, setting $m_i=\alpha_i-k$ and $n_i=\beta_i-\ell$, we have $$x^{\alpha_i}y^{\beta_i}=\{x^ky^\ell, \frac{1}{k(\beta_i-\ell)- \ell(\alpha_i-k)}x^{m_i}y^{n_i}\},$$  which implies that $f=\{x^ky^\ell,g\}$ for some $g\in \O(\mathbb{T}^2)$. Thus, every element of $E_\omega(\mathbb{T}^2)$ can be written as the Poisson bracket of two elements of $E_\omega(\mathbb{T}^2)$. This implies in turn that 
$$\mathcal{E}_{\omega_{\mathbb{T}^2}}(\mathbb{T}^2)=[\mathcal{E}_{\omega_{\mathbb{T}^2}}(\mathbb{T}^2),\mathcal{E}_{\omega_{\mathbb{T}^2}}(\mathbb{T}^2)]$$ 
has width equal to one. 

It remains to show that $\mathcal{E}_{\omega_{\mathbb{T}^2}}(\mathbb{T}^2)$ is a simple Lie algebra, equivalently, that  $(E_{\omega_{\mathbb{T}^2}},\{\cdot, \cdot\})$ is a simple Lie algebra. Assume that $I\subset E_{\omega_{\mathbb{T}^2}}(\mathbb{T}^2)$ is a nonzero ideal. Using \eqref{bracketH_T}
we can show that if $I$  contains a nontrivial monomial $x^my^n$, then it contains all possible nontrivial monomials, hence is equal to $E_{\omega_{\mathbb{T}^2}}(\mathbb{T}^2)$. Indeed,
 up to exchanging $x$ and $y$, we can assume that $m\neq 0$.
 Then bracketing with  $y^{-n+1}$, we get that $$\{y^{-n+1},x^my^n\}=m(-n+1)y^{-n}y^nx^{m-1}\{y,x\}=-m(-n+1)x^my$$ belongs to  
$I$, hence that $x^my\in I$. Bracketing with $x^{-m+1}$, we obtain that $$\{x^{-m+1},x^{m}y\}=(-m+1)x^{-m}x^m\{x,y\}=(-m+1)xy$$
belongs to $I$, hence that $xy\in I$. 
Taking brackets of $xy$ with $y^{-1}$ and $x^{-1}$ respectively, we obtain the elements $x$ and $y$. Since all monomials are obtained from these two taking suitable brackets, %(explicitly, for $k\neq 0$, $x^k=\frac{1}{k}\{x^ky^{-1},y\}$ and if $\ell\neq 0$ then $x^k y^\ell=-\frac{1}{\ell}\{x^{k-1}y^\ell,x\}$, the rest follows by exchanging the roles of $x$ and $y$),  
we conclude that $I=E_{\omega_{\mathbb{T}^2}}(\mathbb{T}^2)$. \\
Now we have to show that any  nonzero ideal $I$ of $E_{\omega_{\mathbb{T}^2}}(\mathbb{T}^2)$ contains a monomial. Let $M$ denote the minimal number of summands of the elements of $I$. If $M=1$, we are done. So assume that $M\geq 2$ and let $f=\sum a_{i,j}x^iy^j$ be an element of $I$ with $M$ summands. Applying similar combinations of brackets as above, we see that $I$ contains an element $q$ with $M$ summands having $x$ as one of its summands. If $q\in k[x^{\pm 1},y^{\pm 1}]\setminus k[x^{\pm 1}]$ then $\{x,q\}$ is a nonzero element of $I$ which has $M-1$ summands, which is impossible. So $q\in k[x^{\pm 1}]$. Then $r=\{q(x),y\}$ is an element of $I$ with $M$ summands whose summands are all linear in $y$. Let $x^{k}y$ be one of the summands of $r$. Then $s=\{x^{-k},r\}$ is an element of $I$ with $M$ summands that has $y$ as one its summands and at least one  summand of the form 
$x^{\alpha}y$ with $\alpha \neq 0$. Then $\{y,s\}$ is a non-zero element of $I$ with at most $M-1$ summands, a contradiction.  
\end{proof}

\subsection{Locally nilpotent vector fields on Danielewski surfaces} 

By a Danielewski surface, we mean an affine surface $\Dp$ in $\mathbb{A}^3$ defined by an equation of the form $xy-p(z)=0$ for some polynomial $p(z)\in k[z]$ of degree $r\geq 1$ with simple roots in $k$. By the Jacobian criterion, $\Dp$ is smooth, and it follows from adjunction formula that   
$$\omega=\frac{dx\wedge dz}{x}|_{\Dp}=\frac{dx\wedge dy}{p'(z)}|_{\Dp}=-\frac{dy\wedge dz}{y}|_{\Dp}$$ 
is a well-defined nowhere vanishing algebraic $2$-form on $\Dp$. 
On the other hand, it immediately follows from the defining equation of $\Dp$ that $\Omega_{\Dp/k}$ is a locally free sheaf or rank $2$ globally generated by the $1$-forms $dx$, $dy$ and $dz$ with the unique relation $xdy+ydx=p'(z)dz$. The following proposition summarizes additional basic properties of the surfaces $\Dp$. 

\begin{proposition} \label{prop:DanBasics1}For a Danielewski surface $\Dp$ with $r=\deg p\geq 1$, the following hold:
\begin{enumerate} 
\item $\O(\Dp)^*=k^*$.
\item $\Omega_{\Dp/k}$ is a free $\O_{\Dp}$-module of rank 2.
\item $\mathrm{H}^1_{\mathrm{dR}}(\Dp)=0$ and $\mathrm{H}^2_{\mathrm{dR}}(\Dp)\simeq k^{\oplus (r -1)}$.
\item If $r\geq 2$, then the class of $\omega$ in $\mathrm{H}^2_{\mathrm{dR}}(\Dp)$ is nonzero.
\end{enumerate}
\end{proposition}
\begin{proof}
The projection $\mathrm{pr_x}\colon\Dp\to \mathbb{A}^1$ is the algebraic quotient morphism of the fixed point free $\mathbb{G}_a$-action associated to the locally nilpotent $k$-derivation $p'(z)\partial_y+x\partial_z$ of $\O(X)$. The fiber $\mathrm{pr}_x^{-1}(0)$ is the union of $r$ disjoint $\mathbb{G}_a$-orbits, one for each of the roots of $p$ whereas all other fibers consist of a single $\mathbb{G}_a$-orbit. The geometric quotient of $\Dp$ is isomorphic to the affine line with $r$ origins $\delta\colon\breve{\mathbb{A}}^1\to \mathbb{A}^1$ obtained from $\mathbb{A}^1$ by replacing the origin $\{0\}$ by $r$ distinct points, one for each of the $\mathbb{G}_a$-orbits in $\mathrm{pr}_x^{-1}(0)$, and the induced morphism $\rho\colon\Dp\to \breve{\mathbb{A}}^1$ is a $\mathbb{G}_a$-torsor. Assertion (1) then follows from the fact that the pullback homomorphism $\rho^*\colon k^*=H^0(\breve{\mathbb{A}^1},\O_{\breve{\mathbb{A^1}}}^*)\to \O(\Dp)^*$ is an isomorphism. 

Since $\rho$ is a $\mathbb{G}_a$-torsor, we have $\Omega^1_{\Dp/\breve{\mathbb{A}^1}}\cong \rho^*\O_{\breve{\mathbb{A}^1}}=\O_{\Dp}$, and since $\delta\colon\breve{\mathbb{A}}^1\to \mathbb{A}^1$ is an \'etale morphism, we have $\Omega_{\breve{\mathbb{A}}^1/k}\cong \delta^*\Omega_{\mathbb{A}^1/k}\cong \O_{\breve{\mathbb{A}}^1}$. The relative cotangent sequence of $\rho$ thus reads 
$$0\to \rho^*\Omega_{\breve{\mathbb{A}}^1/k}=\O_{\Dp}\to \Omega_{\Dp/k}\to \Omega_{\Dp/\breve{\mathbb{A}}^1}=\O_{\Dp} \to 0.$$
Since $\Dp$ is affine, the latter splits, and the choice of such a splitting yields an isomorphism  $\Omega_{\Dp/k}\cong \O_{\Dp}^{\oplus 2}$ which proves (2).   

To prove assertion (3), we proceed by induction on the degree $r$ of $p$. If $r=1$, then $\Dp\cong \mathbb{A}^2$, and the assertion is clear. Now assume that $r\geq 2$. Without loss of generality, we can further assume that $p(z)=zq(z)$ for some polynomial $q$ of degree $r-1$ with simple roots and such that $q(0)\neq 0$. Then $\mathrm{pr}_x^{-1}(0)$ is the disjoint union of the curves $L_0=\{x=z=0\}$ and $L_q=\{x=q(z)=0\}$. Set $U_0=\Dp\setminus L_q$ and $U_q=\Dp\setminus L_0$.  The Mayer--Vietoris long exact sequence of algebraic de Rham cohomology for the covering of $\Dp$ by the affine open subsets $U_0$ and $U_q$ yields the exact sequence  
\[
\begin{array}{l}
0\to \mathrm{H}^1_{\mathrm{dR}}(\Dp)\to \mathrm{H}^1_{\mathrm{dR}}(U_0)\oplus \mathrm{H}^1_{\mathrm{dR}}(U_q)\to 
\mathrm{H}^1_{\mathrm{dR}}(U_0\cap U_q) \to \cdots \\
\\
\to  \mathrm{H}^2_{\mathrm{dR}}(\Dp)\to \mathrm{H}^2_{\mathrm{dR}}(U_0)\oplus \mathrm{H}^2_{\mathrm{dR}}(U_q)\to \mathrm{H}^2_{\mathrm{dR}}(U_0\cap U_q)\to 0.
\end{array}
\]
Since $U_0\cap U_q=\Dp\setminus\mathrm{pr_x}^{-1}(0)\cong \mathrm{Spec}\,k[x^{\pm 1},z]=(\mathbb{A}^1\setminus\{0\})\times \mathbb{A}^1$, we have $\mathrm{H}^2_{\mathrm{dR}}(U_0\cap U_q)=0$. Furthermore, the group $\mathrm{H}^1_{\mathrm{dR}}(U_0\cap U_q)$ is isomorphic to $k$, generated
by the class of the $1$-form $\frac{dx}{x}$. Since on the other hand we have isomorphisms $$\mathbb{A}^2\stackrel{\cong}{\to} U_0,\;(x,u)\mapsto (x,q(x,u),xu) \quad  \textrm{and} \quad \Dq\stackrel{\cong}{\to} U_q, \; (x,y,z)\mapsto (x,yz,z),$$
assertion (3) follows by induction. 

To prove assertion (4), we can assume without loss of generality that $p(z)=z(z-1)s(z)$ for some polynomial $s$ or degree $r-2$ with simple roots and such that $s(0)s(1)\neq 0$. Denote by $U_s$ the affine open complement in $\Dp$ of the curve $\{x=s(z)=0\} \subset \Dp$. It suffices to show that $\omega|_{U_s}$ has non-zero class in $\mathrm{H}^2_{\mathrm{dR}}(U_s)$. 
Since the morphism $$f\colon D_{z(z-1)}\to \Dp, \; (x,y,z)\mapsto (x,y,z)\mapsto (x,q(z)y,z)$$ is an open immersion with image $U_s$, 
it is equivalent to show that the class of the volume form $f^*\omega$ in $\mathrm{H}^2_{\mathrm{dR}}(D_{z(z-1)})\cong k$ is non-zero. The surface $S=D_{z(z-1)}$ is covered by the two affine open subsets 
$$ S_{0}=S\setminus \{x=z-1=0\}\cong \mathrm{Spec}\,k[x,u_0] \; \textrm{and} \; S_{1}=S\setminus \{x=z=0\}\cong \mathrm{Spec}\,k[x,u_1]$$ where $u_0=z/x=y/(z-1)$ and $u_1=(z-1)/x=y/z$. Under the above isomorphisms,  $f^*\omega|_{S_0}$ and $f^*\omega|_{S_1}$ are equal to the volume forms $dx\wedge du_0$ and $dx\wedge du_1$,  respectively. Since $S_0$ and $S_1$ are both isomorphic to $\mathbb{A}^2$, the connecting homomorphism $$\delta\colon  \mathrm{H}^1_{\mathrm{dR}}(S_0\cap S_1)\to \mathrm{H}^2_{\mathrm{dR}}(S)$$ is an isomorphism. Noting that $u_1=u_0+\frac{1}{x}$, we see that the class of $f^*\omega$ is equal to the image under $\delta$ of the class of the $1$-form $\frac{dx}{x}$ on $S_0\cap S_1\cong \mathrm{Spec}\,k[x^{\pm 1},z]$. Since the latter is a generator of $\mathrm{H}^1_{\mathrm{dR}}(S_0\cap S_1)\cong k$, assertion (4) follows.
\end{proof}

It follows from Proposition \ref{prop:DanBasics1}(3) that for the symplectic affine surface $(\Dp,\omega)$, we have  $\mathrm{VP}_\omega(\Dp)=\mathcal{H}_\omega(\Dp)$, and hence that $\langle \mathrm{LNV}(\Dp)\rangle=\langle \mathcal{H}_\omega\mathrm{LNV}(\Dp)\rangle$. In the natural decomposition of $\O(\Dp)$ into a direct 
sum of $k$-vector spaces 
\begin{equation}\label{formula5}
\O(\Dp) = xk[x,z] \oplus yk[y, z] \oplus k[z],
\end{equation}
the associated Poisson bracket on  $\O(\Dp)$ is uniquely determined by the three values 
\begin{equation}\label{PoissonDp}
 \{x,z\}=x, \{x,y\}=p'(z),\; \textrm{and} \; \{y,z\}=-y
 \end{equation}
and extension by linearity and the Leibniz rule. 

\subsubsection{Locally nilpotent vector fields on Danielewski surfaces}
The tangent sheaf $\mathcal{T}_{\Dp}$ of a Danielewski surface $\Dp$ is a free $\mathcal{O}_{\Dp}$-module of rank $2$ globally generated as an $\mathcal{O}_{\Dp}$-module by the vector fields 
\begin{equation}\label{formula1}
 \theta_x=p'(z) \frac{\partial}{\partial y} + x \frac{\partial}{\partial z}, \; \theta_y:=-p'(z) \frac{\partial}{\partial x} - y \frac{\partial}{\partial z}, \;
\theta_z=-x\frac{\partial}{\partial x} +y \frac{\partial}{\partial y},
\end{equation}
(see \eqref{PoissonIso} for the notation) with the unique relation $x\theta_y + y\theta_x = p'(z)\theta_z$. Note that the vector fields $\theta_x$ and $\theta_y$ are locally nilpotent and belong to $\mathcal{E}_\omega(\Dp)$. 
%The same holds for $\theta_{f(x)} = f'(x)\theta_x$ and $\theta_{f(y)} = f'(y)\theta_y$. 
%Moreover, $[\theta_x,\theta_y] = \theta_{\{x,y\}_\omega}= p''(z)\theta_z$, and so $p''(z)\theta_z \in \lielnd{\Dp}\cap  \mathcal{Z}_\omega(\Dp)$.

\medskip

The following result is an extension of \cite[Theorem 3.26]{KL13} and \cite[Theorem 1]{LR19}. 

\begin{proposition} \label{LieDpEqualities}For a Danielewski surface $\Dp$, the following hold
\begin{enumerate}
    \item $
E_\omega(\Dp) =xk[x,z]\oplus yk[y,z]\oplus \{ (r(z)p(z))' \mid r\in k[z] \}$.
    \item $\mathcal{E}_\omega(\Dp)=[\mathcal{H}_\omega(\Dp),\mathcal{H}_\omega(\Dp)]=\lielnd{\Dp}$. Furthermore, this Lie algebra is simple. 
\end{enumerate}
\end{proposition}
\begin{proof}
By \eqref{formula5}, every  element $e\in \O(\Dp)$ admits a unique expression in the form  $$e=xe_x(x,z)+ye_y(y,z)+e_z(z).$$ For every $f,g,h \in \O(\Dp)$, we have 
%\[\begin{array}{rcl}
%\mathcal{L}_{f\theta_x+g\theta_y+h\theta_z}\omega & = &df\wedge dx+dg\wedge dy+dh\wedge dz \\
%& = & (-\frac{\partial f}{\partial y}+\frac{\partial g}{\partial x}) dx\wedge dy+ (-\frac{\partial f}{\partial z}+\frac{\partial h}{\partial x}) dx\wedge dz+ (-\frac{\partial g}{\partial z}+\frac{\partial h}{\partial y})dy\wedge dz \\
 %& = & \left( p'(z)(-\frac{\partial f}{\partial y}+\frac{\partial g}{\partial x})+x(-\frac{\partial f}{\partial z}+\frac{\partial h}{\partial x})-y(-\frac{\partial g}{\partial z}+\frac{\partial h}{\partial y}) \right)\omega
%\end{array}
%\]
%and hence,
\[
\begin{array}{rcl}
\mathrm{Div}_\omega(f\theta_x+g\theta_y+h\theta_z) & =& p'(z)(-\frac{\partial f}{\partial y}+\frac{\partial g}{\partial x})+x(-\frac{\partial f}{\partial z}+\frac{\partial h}{\partial x})-y(-\frac{\partial g}{\partial z}+\frac{\partial h}{\partial y}) \\
& = & xu(x,z)+yv(y,z)+ ((g_x(0,z)-f_y(0,z))p(z))'.
\end{array}
\]
Every element of $\O(\Dp)$ of the form $xu(x,z)$ (resp. $yv(y,z)$ is equal to $\mathrm{Div}_\omega(f\theta_x)$ (resp. $\mathrm{Div}_\omega(g\theta_y)$) for some element $f$ of the form $xr(x,z)$ (resp. $g$ of the form $yr(y,z))$). Since $H^0(\Dp,\mathcal{T}_{\Dp})$ is generated as an $\O(\Dp)$-module by $\theta_x$, $\theta_y$ and $\theta_z$, we conclude that $E_\omega(\Dp)\subset \O(\Dp)$ consists of all elements whose component in $k[z]$ is of the form $(r(z)p(z))'$ for some polynomial $r(z)\in k[z]$.

Using the decomposition \eqref{formula5}, the definition \eqref{PoissonDp} and the identity \[\begin{array}{rcl}
\{xf(z),yg(z)\} & = & xg(z)\{f(z),y\}+yf(z)\{x,g(z)\}+f(z)g(z)\{x,y\} \\
& = & xg(z)f'(z)\{z,y\}+ yf(z)g'(z)\{x,z\}+f(z)g(z)\{x,y\} \\
& = & xy(f(z)g(z))'+f(z)g(z)p'(z) \\
& = & (f(z)g(z)p(z))',
\end{array}
\]
it is then straightforward to check that $\{\O(\Dp),\O(\Dp)\}=E_\omega(\Dp)$, and hence, by definition, that $\mathcal{E}_\omega(\Dp)=[\mathcal{H}_\omega(\Dp),\mathcal{H}_\omega(\Dp)]$.

The equality $\mathcal{E}_\omega(\Dp)=\lielnd{\Dp}$ follows from the proof of Theorem 3.26 in \cite{KL13}, given there over $k=\mathbb{C}$,  but which remains valid over any algebraically closed field of characteristic zero. Finally, the simplicity of the Lie algebra $\lielnd{\Dp}$ is established over $k=\mathbb{C}$ in \cite{LR19}, but again the proof carries on verbatim to the case of an arbitrary algebraically closed field of characteristic zero. 
\end{proof}

Let $\Dp$ be a Danielewski surface. If $\deg p=1$, then  the projection $\mathrm{pr}_{x,y}\colon\Dp\to \mathbb{A}^2$ is a symplectomorphism between
 $(\Dp,\omega)$ and $(\mathbb{A}^2, dx\wedge dy)$.  By Example \ref{ex:sympA2}, the bracket width of  $\lielnd{\Dp}\cong \mathcal{H}_{dx\wedge dy}(\mathbb{A}^2)$ equals one. 
 
\begin{corollary} Let $\Dp$ be a Danielewski surface with $\deg p\geq 2$. Then $\lielnd{\Dp}$ is a Lie ideal of codimension $\deg p -2$ of $\mathcal{H}_\omega(\Dp)$, with quotient isomorphic to $\mathrm{H}^2_{\mathrm{dR}}(\Dp)/\langle \overline{\omega} \rangle $.
\end{corollary}
\begin{proof}
By Proposition \ref{prop:DanBasics1}(3) and Proposition \ref{LieDpEqualities}, $E_\omega(\Dp)$ is a proper ideal of codimension $(\deg p-1)$ 
of the Lie algebra $(\O(\Dp),\{\cdot, \cdot\})$.  Since by  Proposition \ref{prop:DanBasics1}(4) the class $\overline{\omega}$ of the $2$-form $\omega$ in $\mathrm{H}^2_{\mathrm{dR}}(\Dp)$
is nonzero, Proposition \ref{LieDpEqualities} implies that $\lielnd{\Dp}=\mathcal{E}_\omega(\Dp)\cong E_{\omega}(\Dp)/k$ is a Lie ideal of $\mathcal{H}_\omega(\Dp)\cong \O(\Dp)/k$ with quotient isomorphic to $\mathrm{H}^2_{\mathrm{DR}}(\Dp)/\langle \overline{\omega} \rangle $.
\end{proof}

The case where $\deg p=2$ is special since we have $\mathcal{H}_\omega(\Dp)=\lielnd{\Dp}=\mathcal{E}_\omega(\Dp)$. 
\begin{proposition}\label{danielewskiwidthzero} Let $\Dp$ be a Danielewski surface with $\deg p=2$. Then the width of the simple Lie algebra $\lielnd{\Dp}$ is at most two.
\end{proposition}
\begin{proof}
The assertion is equivalent to the property that every element of $E_\omega(\Dp)$ can be written as the sum of at most two Poisson brackets of elements of $E_\omega(\Dp)$.  
Since $\deg p=2$, it follows from Proposition \ref{LieDpEqualities}(1) that $E_\omega(\Dp)$ contains a polynomial $z+a$ for some $a\in k$. First note that since $\{x,z+a\}=x$ and $\{y,z+a\}=-y$, every element
\[
 f=\sum_{i>0}a_i(z)x^i + \sum_{j > 0}b_j(z)y^j  \in E_\omega(\Dp)
 \]
can be written as $f=\{g,z+a\}$,  where 
 \[
 g =   \sum_{i>0} \frac{1}{i}a_i(z)x^{i} - \sum_{j > 0} \frac{1}{j}b_j(z)y^{j} \in E_\omega(\Dp).
 \]
On the other hand, since $\{x,y\}=p'(z)$, every element of the form $(r(z)p(z))'\in E_\omega(\Dp)$ is equal to the bracket $\{x,yr(z)\}$. Thus,  every element of $E_\omega(\Dp)$ can be written  as the sum $\{g,z+a\}+\{x,yr(z)\}$ of two brackets of elements of $E_\omega(\Dp)$.  
%First note that $x\d_x - y \d_y = \frac{1}{2} [\nu_x,\nu_y] \in \lielnd{\Dp}$. Hence, for any $f$ of the form
% such that $(x\d_x - y \d_y)(g)=f$. Therefore, by Proposition \ref{prop10} we have 
% \begin{equation}\label{thetaf}
% \theta_f = \theta_{(x\d_x - y \d_y)(g)} = [x\d_x - y \d_y,\theta_g].
% \end{equation}
% Further, $(r(z)p(z))'$ can be represented as 
% \[
%  \nu_x(r(z)y)=(p'(z) \frac{\partial}{\partial y} + x \frac{\partial}{\partial z})(r(z)y) = (r(z)p(z))'
% \]
% and hence
% \begin{equation}\label{theta(rp)}
% \theta_{(r(z)p(z))'} = %[\nu_x,\theta_{r(z)y}].
% \end{equation}
% By Proposition \ref{prop9}, any $\mu \in \lielnd{\Dp}$ can be represented as a sum 
% of
% $\theta_{f}$ from \eqref{thetaf} and $\theta_{(r(z)p(z))'}$ from \eqref{theta(rp)}, and the assertion of 
 % the proposition follows.
\end{proof}

We did not succeed to decide whether the width of $\lielnd{\Dp}$ for $\deg p=2$ is equal to one or to two. We note the following equivalent formulations:

\begin{lemma} \label{biexact}
For a Danielewski surface  $\Dp$ with $\deg p = 2$, the following are equivalent:
\begin{enumerate}
    \item The width of the Lie algebra $\lielnd{\Dp}$ equals one. 
    \item Every element of $E_\omega(\Dp)$ equals the Poisson bracket $\{f,g\}$ of two elements of $\O(\Dp)$.
    \item Every algebraic $1$-form $\alpha\in H^0(\Dp,\Omega^1_{\Dp/k})$ can be written as $\alpha=fdg +dh$ for some $f,g,h\in \O(\Dp)$.
\end{enumerate}
\end{lemma}
\begin{proof}
The equivalence between (1) and (2) follows from the isomorphism $(\O(\Dp)/k,\{\cdot,\cdot\})\cong (\mathcal{H}_\omega(\Dp),[\cdot,\cdot])$ and the equality $\mathcal{E}_\omega(\Dp)=\lielnd{\Dp}=\mathcal{H}_\omega(\Dp)$ which holds by Proposition \ref{LieDpEqualities}. Since elements $e$ of $E_\omega(\Dp)$ are precisely those of $\O(\Dp)$ such that the algebraic $2$-form $e\omega$ is exact, say $e\omega=d\alpha$, it follows from the definition of the Poisson bracket that $e=\{f,g\}$ if and only if $e\omega=df\wedge dg=d(fdg)=d\alpha$. The equivalence between (2) and (3) follows.  
\end{proof}

\begin{question}\label{q:Dp2}
Is the the width of the Lie algebra $\lielnd{\Dp}$ of a Danielewski surface $\Dp$ with $\deg p =2$ equal to two? 
\end{question}

%\medskip

We now consider the case of Danielewski surfaces $\Dp$ with $\deg p \geq 3$. The  projection  $\pi=\mathrm{pr}_{x,z}\colon\Dp\to \mathbb{A}^2$ restricts to an isomorphism $\Dp\setminus\{x=0\} \cong \mathbb{A}^2\setminus \{x=0\}$. Every pair of elements $f,g\in \O(\Dp)$ determines a rational map  $$h_{f,g}=(f,g)\circ \pi^{-1}:\mathbb{A}^2 \dashrightarrow \mathbb{A}^2.$$ 
Denote by $\Jac(h_{f,g})\in k[x^{\pm 1},z]$ the unique element such that $h_{f,g}^*\omega_{\mathbb{A}^2}=\Jac(h_{f,g})\omega_{\mathbb{A}^2}$. Viewing $f$ and $g$ as elements of $\O(\Dp)_x\cong k[x^{\pm 1},z]$ via the injective localization homomorphism $\O(\Dp)\hookrightarrow \O(\Dp)_x$ with image $k[x,z,x^{-1}p(z)]$, we have $$\Jac(h_{f,g})=\frac{\partial f(x,x^{-1}p(z),z)}{\partial z}\frac{\partial g(x,x^{-1}p(z),z)}{\partial x}-\frac{\partial f(x,x^{-1}p(z),z)}{\partial x}\frac{\partial g(x,x^{-1}p(z),z)}{\partial z}.$$  
We introduce the following condition reminiscent to the Jacobian conjecture for $\mathbb{A}^2$.

\begin{hyp} If $\Jac (h_{f,g})\in k^*$ then $h_{f,g}$ is a biregular automorphism of $\mathbb{A}^2$. 
\end{hyp}

\begin{theorem}\label{main1} Let $\Dp$ be a Danielewski surface with $\deg p \geq 3$. Assume that Hypothesis $\text{\rm{(J)}}$ holds for $\Dp$. Then 
the width of the simple Lie algebra $\lielnd{\Dp}$ is greater than one.
\end{theorem}

\begin{proof}
We will show that the locally nilpotent vector field $\theta_{x}$ cannot be represented as a single bracket of elements of $\lielnd{\Dp}$. By Proposition  \ref{LieDpEqualities}, this is equivalent to showing that the element $x\in E_\omega(\Dp)$ cannot be written as the Poisson bracket of two elements of $E_\omega(\Dp)$. 

Assume to the contrary that there exist $f,g\in E_\omega(\Dp)$ such that $x=\{f,g\}$. Then by combining the definition of the Poisson bracket with the equality $(f,g)=h_{f,g}\circ \pi$, we obtain:
\[
\begin{array}{rcl}
x\omega= df\wedge dg=(f,g)^*\omega_{\mathbb{A}^2} & = &   \pi^*h_{f,g}^*\omega_{\mathbb{A}^2} \\
& = &\pi^*(\Jac(h_{f,g})\omega_{\mathbb{A}^2})\\
& =& (\pi^*\Jac(h_{f,g}))dx\wedge dz=(\pi^*\Jac(h_{f,g}))x\omega.
\end{array}
\]
Since $\pi\colon\Dp\to \mathbb{A}^2$ is dominant, it follows that $\Jac(h_{f,g})=1$ and hence, by hypothesis (J), that $h_{f,g}$ is a biregular automorphism of $\mathbb{A}^2$. The composition $h_{f,g}\circ \pi$ then induces an isomorphism between $\O(\Dp)$ and the sub-algebra $k[f,g,f^{-1}p(g)]$ of $k[f^{\pm 1},g]$. This implies in turn that the map $\varphi\colon\Dp\to \Dp$, $(x,y,z)\mapsto (f,f^{-1}p(g),g)$,  is an automorphism of $\Dp$ which makes the following diagram commutative
\begin{equation}
\begin{tikzcd}
  \Dp \arrow{d}[swap]{\varphi} \arrow{r}{\pi} & \mathbb{A}^2 \arrow{d}{h_{f,g}} \\
  \Dp \arrow{r}{\pi} & \mathbb{A}^2.
 \end{tikzcd}
 \end{equation}
By construction, we have  $\varphi^*g=z$, where, by hypothesis, $g$ is an element of $E_\omega(\Dp)$. Recall that by definition of $E_\omega(\Dp)$, this means that $g=\mathrm{Div}_\omega\xi$ for some $\xi\in \Ve(\Dp)$, equivalently that $d(i_\xi\omega)=g\omega$. By Proposition \ref{prop:DanBasics1}(1), there exists $\lambda \in k^*$ such that $\varphi^*\omega=\lambda \omega$.  
Let $d\varphi\colon\mathcal{T}_{\Dp}\to \varphi^*\mathcal{T}_{\Dp}$ be the tangent homomorphism to $\varphi$, and let $\xi'=(d\varphi)^{-1}(\varphi^*\xi)\in \Ve(\Dp)$ be the pullback of $\varphi^*\xi\in H^0(\Dp,\varphi^*\mathcal{T}_{\Dp})$. We then have 
\[
\begin{array}{rcl}
 (\mathrm{Div}_\omega  \xi') \omega=d(i_{\xi'}\omega) & = & d(\varphi^*(i_\xi( (\varphi^{-1})^*\omega)))=d(\varphi^*(i_\xi(\lambda^{-1}\omega)))\\
 & = & d(\varphi^*(\lambda^{-1}i_\xi\omega)) = \varphi^*(\lambda^{-1}d(i_\xi \omega)) \\
 & =& \varphi^*(\lambda^{-1} g\omega) =  (\varphi^*g) \omega \\ 
 &= & z\omega.
 \end{array}
\]
Thus, $z\in E_\omega(\Dp)$. Since $\deg p\geq 3$, this is impossible because, by Proposition \ref{LieDpEqualities}, every element of $E_\omega(\Dp)\cap k[z]$ is of the form $(r(z)p(z))'$ for some $r(z)\in k[z]$. 
\end{proof}

\section{Concluding remarks}

As the reader may have noticed, the results presented above reflect only first steps in understanding the width of simple Lie algebras. In this short section, we summarize our vision of eventual forthcoming steps to be considered. 
For brevity, below we shorten `algebras of width greater than one' to `wide algebras'. 

The first immediate question to ask is 
\begin{question}
What is the actual width of the algebras appearing 
in Theorems A and B? 
Can it be made as large as possible?
\end{question}

For better understanding of the width behaviour, it is highly 
desirable to enlarge the bank of examples of wide simple Lie algebras. 
A natural way (suggested by \name{Yuly Billig}) for generalizing the examples of Theorem A  is towards the  Krichever--Novikov algebras 
related to vector fields on punctured curves 
\cite{Sch}. 
This class of infinite dimensional algebras is extremely rich. These algebras arise from meromorphic objects (functions,
vector fields, forms of certain weights, matrix valued functions, etc.) which are holomorphic
outside a fixed set of points. This construction includes Lie algebras and associative algebras
(and also superalgebras, Clifford algebras, etc.). In some natural situations algebras belonging to this family are known to be simple (see \cite[Proposition~6.99]{Sch}) and look as a source for eventual wide algebras.

Note that the Lie algebra $L=\Ve (X)$ captures the geometry of $X$ in the sense
that the isomorphism of Lie algebras $\Ve (X)\cong \Ve (Y)$ implies
that $X$ and $Y$ are isomorphic (here $X$ and $Y$ are arbitrary normal varieties), see
\cite{Gra}, \cite{Sie96}. This gives rise to the following vague question. 

\begin{question}
What geometric properties of $X$ imply that $L=\Ve (X)$ is wide?
\end{question}

Pursuing the geometric flavour of the notion of width, 
one can ask 

\begin{question}
Does there exist a Lie-algebraic counterpart of the Barge--Ghys example from \cite{BG} mentioned in the introduction? 
\end{question}

This requires going over to the
category of smooth vector fields on smooth manifolds. 
Note that even simpler looking problems discussed in \cite{LT} are not yet settled being related to subtle 
differential-geometric considerations. 

\medskip

One can also look for additional sources of wide simple Lie algebras.  
A possible candidate (also suggested by \name{Yuly Billig}) is the following one:

\begin{question}
Let $K_2$ denote the Lie algebra obtained from the matrix
$
\left(
\begin{matrix}
 2 & 2 \\ 2 & 2
\end{matrix}
\right)
$
in the same way as Kac--Moody Lie algebras are obtained from generalized Cartan matrices,
see \cite[\S 6]{Kac}. Is $K_2$ wide?
\end{question}

Here is a challenging general question. 

\begin{question}
Do there exist simple Lie algebras of infinite width?
\end{question}

Note that in sharp contrast with the group case, where there 
are examples of finitely generated simple groups of infinite width, the width of any finitely generated Lie algebra is 
finite, see \cite{Rom}. 

\medskip

Finally, one can ask a `metamathematical' question.

\begin{question}
Let $L$ be a `generic' (`random', `typical') simple Lie algebra. Is $L$ wide?
\end{question}

Of course, any eventual answer will heavily depend on what is meant by the euphemisms used in the statement.  
However, the absence of semisimple and nilpotent elements in the Lie algebra $\Ve(C)$ appearing in Theorem A (see Corollary  \ref{no-ss}) is a witness of the absence of any analogue of the triangular decomposition. This gives some evidence for the  following (`metamathematical') working hypothesis: `amorphous' (less structured) simple Lie algebras tend to be wide. In a sense, this is supported by the cases of the  algebras of Cartan type and $\mathrm{VP}_{\omega_{\mathbb{A}^n}}(\mathbb A^n)$ whose automorphisms groups were computed in \cite{Rud} and \cite{Bav17}, see also \cite{KR17}, respectively.

\medskip

\noindent{\it Acknowledgements}. We thank Joseph Bernstein, Yuly Billig, Mikhail Borovoi, Zhihua Chang, Be'eri Greenfeld, Hanspeter Kraft, Leonid Makar-Limanov, Anatoliy Petravchuk, Vladimir Popov, Oksana Yakimova, and Efim Zelmanov for useful discussions regarding various aspects of this work.

%\vskip1cm

\end{document}